\documentclass[USenglish,reqno,intlimits]{amsart}
\usepackage{amssymb, amsthm, amscd, mathrsfs}

\usepackage[affil-it]{authblk} 
\usepackage[foot]{amsaddr}

\usepackage{paralist}

\edef\restoreparindent{\parindent=\the\parindent\relax}
\usepackage{parskip}
\restoreparindent

\usepackage{bigints}
\usepackage{empheq}
\usepackage{cases}
\usepackage{enumitem}
\usepackage{dsfont}
\usepackage{cite}
\usepackage{float}

\usepackage{hyperref}
\makeatletter
\def\specialsection{\@startsection{section}{1}%
  \z@{\linespacing\@plus\linespacing}{.5\linespacing}%
  {\raggedright\sffamily\LARGE\bfseries}}% NEW
\def\section{\@startsection{section}{1}%
  \z@{.9\linespacing\@plus\linespacing}{.7\linespacing}%
  {\raggedright\sffamily\LARGE\bfseries}}% NEW
  \renewcommand{\@secnumfont}{\raggedright\sffamily\Large\bfseries}
\makeatother

\makeatletter
\def\subsection{\@startsection{subsection}{2}%
  \z@{.9\linespacing\@plus.7\linespacing}{.7\linespacing}%
  {\raggedright\sffamily\Large\bfseries}}
\makeatother
\newtheoremstyle{dthm}
  {.5\baselineskip}
  {.5\baselineskip}
  {\itshape}
  {}
  {\sffamily\bfseries}
  {\ifx\thmnote\@gobble.\else\normalfont.\fi}
  {.5em}
  {}
\newtheoremstyle{ddef}
  {.5\baselineskip}
  {.5\baselineskip}
  {\normalfont}
  {}
  {\sffamily\bfseries}
  {\ifx\thmnote\@gobble.\else\normalfont.\fi}
  {.5em}
  {}
\theoremstyle{dthm}
\newtheorem{theorem}{Theorem}
\newtheorem{corollary}{Corollary}
\newtheorem{lemma}{Lemma}
\newtheorem{proposition}{Proposition}
\theoremstyle{ddef}
\newtheorem{definition}{Definition}
\newtheorem{remark}{Remark}

\def \dv {\mathrm{div}}
\def \d {\mathrm{d}}

\makeatletter
\def\@settitle{\begin{center}%
  \baselineskip14\p@\relax
    \raggedright\sffamily\Huge\bfseries%<- NEW
\uppercasenonmath\@title
  \@title
  \end{center}%
}
\makeatother
\title[I\MakeLowercase{nverse problems for general parabolic systems and application}] %Use the shortened version of the full title
      {I\MakeLowercase{nverse problems for general parabolic systems and application to} O\MakeLowercase{rnstein}-U\MakeLowercase{hlenbeck equation}}

\address{\textsf{\textbf{E. M. Ait Ben Hassi, S. E. Chorfi, L. Maniar,} Department of Mathematics, Faculty of Sciences Semlalia, LMDP, UMMISCO (IRD-UPMC), Cadi Ayyad University, Marrakesh 40000, B.P. 2390, Morocco.}}

\email{m.benhassi@gmail.com}	
\email{chorphi@gmail.com}	
\email{maniar@uca.ma}

\patchcmd{\abstract}{3pc}{0pt}{}{}

\begin{document}
\begin{abstract}
\normalsize
We investigate the link between inverse problems and final state observability for a general class of parabolic systems. We generalize a stability result for initial data due to García and Takahashi \cite{GaT'11}, known for the case of self-adjoint dissipative operators. More precisely, we consider a system governed by an analytic semigroup. Under the assumption of final state observability, we prove a logarithmic stability estimate depending on the analyticity angle of the semigroup. This is done by using a general logarithmic convexity result. The abstract result is illustrated by considering the Ornstein–Uhlenbeck equation.

Ornstein–Uhlenbeck equation

\end{abstract}

\noindent \textsf{\Large E. M. Ait Ben Hassi, S. E. Chorfi, and L. Maniar}
\vspace{-0.1cm}

\maketitle

\section{Introduction}
Let $H,Y$ be Hilbert spaces and $A \colon D(A)\subset H \rightarrow H$ be the generator of an analytic semigroup $\left(\mathrm{e}^{tA}\right)_{t\ge0}$ of angle $\psi \in \left(0,\dfrac{\pi}{2}\right]$. We denote by $H_A$ the space $D(A)$ endowed with the graph norm. Consider the following observation system
\begin{empheq}[left = \empheqlbrace]{alignat=2}
\begin{aligned}
& u'(t)=A u(t), \quad t \in (0,\theta],\\
& u(0)=u_0 \in H, \label{E1}
\end{aligned}
\end{empheq}
where $\theta>0$ is a final time for the system. Consider an admissible observation operator $\mathbf{C}\in \mathcal{L}(H_A, Y)$ for $\left(\mathrm{e}^{tA}\right)_{t\ge0}$. That is, there exists a constant $\kappa_{\theta}> 0$ such that
\begin{equation}
\forall u_0\in D(A), \qquad \int_{0}^{\theta}\left\|\mathbf{C} \mathrm{e}^{t A}u_0\right\|_{Y}^{2}\,\mathrm{d}t \le \kappa_\theta^2 \|u_0\|_H^2. \label{adm1}
\end{equation}
Recall that the system \eqref{E1} is final state observable in time $\theta$ if there exists a constant $\kappa_{\mathrm{obs}}(\theta)> 0$ such that the following observability inequality holds
\begin{equation}
\forall u_0\in D(A), \qquad \left\|\mathrm{e}^{\theta A}u_{0}\right\|_{H}^2 \leq \kappa_{\text{obs}}^2\int_{0}^{\theta}\left\|\mathbf{C} \mathrm{e}^{t A}u_0\right\|_{Y}^{2}\,\mathrm{d}t. \label{obs1}
\end{equation}
It is well-known that the notion of final state observability is equivalent to null-controllability of the predual system, see, for instance, \cite{TW'09,Za'20}.

\textbf{Inverse initial data problem.} 
For fixed constants $\epsilon \in (0,1)$ and $M>0$, we denote the set of admissible initial data by
\begin{align}
\mathcal{I}_{\epsilon,M}:=\{u_0\in D((-A)^\epsilon) \colon \|u_0\|_{D((-A)^\epsilon)} \le M\}. \label{init1}\
\end{align}
Our purpose is to determine initial conditions $u_0$ belonging to $\mathcal{I}_{\epsilon,M}$ from the measurement $\mathbf{C}u(t)$ in $(0,\theta)$ and obtain a conditional logarithmic stability estimate. Conditional stability in this context means that a priori bound is imposed on the initial data which are taken in an admissible set such as $\mathcal{I}_{\epsilon,M}$. It is worth mentioning that conditional stability estimates are very useful when dealing with numerical reconstruction of initial data \cite{LYZ'09, YZ'01}.

The prototypical example of system \eqref{E1} within the framework of observability and inverse problems is the heat equation in $L^2(\Omega)$, with $\Omega \subset \mathbb{R}^N$ is an open bounded domain. Namely, the case when $A =\Delta_\mathrm{D}$ is the Dirichlet Laplacian, which is a self-adjoint operator. The observation is taken either in a small interior non-empty open set $\omega\subset \Omega$ or on a part of the boundary $\Gamma_0 \subset \partial \Omega$. i.e., $\mathbf{C}u=\mathds{1}_\omega u$ or $\mathbf{C}u=\partial_\nu u\rvert_{\Gamma_0}$.

The observability and inverse problems for such equation have been studied since the pioneering works by Bukhgeim and Klibanov \cite{BK'81}, and Fursikov and Imanuvilov \cite{FI'96}. The connection between the two problems was first remarked by Puel and Yamamoto \cite{PY'96} and then investigated by many other authors \cite{ASTT'09, ACT'19, GaT'11, YZ'08}. In \cite{ASTT'09}, the authors have established stability estimate for a source term for exactly observable systems. The abstract results have been applied to the Euler-Bernoulli plate equation. The survey paper \cite{ACT'19} reviews some recent results on solving inverse problems by using observability estimates.

In \cite{YZ'08}, Yamamoto and Zou have proven a logarithmic stability result for the inverse initial data problem for the standard heat equation using a final observability estimate for $u\rvert_{(\tau, \theta)\times \omega}$ where $\tau>0$. The authors used the method of logarithmic convexity from \cite{Pa'75} for the operator $\Delta_\mathrm{D}$. Using the same ideas, similar results were obtained in a more abstract framework for self-adjoint dissipative operators in Hilbert spaces \cite{GaT'11} with application to Stokes systems and a linear fluid-structure system. The logarithmic convexity result used in \cite{GaT'11, YZ'08} works only for self-adjoint operators. There is an extension of this method to some parabolic inequalities (see \cite[Theorem 3.1.3]{Is'17}), but it does not include an inequality of the form
\begin{equation}
|u_t - \Delta u| \le C |\nabla u|, \label{ineq}
\end{equation}
where $C>0$ is a constant. On the other hand, when $\Omega \subset \mathbb{R}^n$ is a bounded domain, Klibanov \cite{KL'06} proved a logarithmic stability estimate for initial data in general parabolic inequalities similar to \eqref{ineq} using a boundary observation. The proof merely relies on Carleman estimates. The case of an unbounded domain $\Omega \subseteq \mathbb{R}^N$ was studied in \cite{KT'07}. Although the strategy in \cite{KL'06, KT'07} is quite general and actually holds for a general second order elliptic operator with $t$-dependent coefficients, it does not cover elliptic operators with unbounded gradient coefficients, since inequality \eqref{ineq} does not hold in this case.

From another viewpoint, the analyticity angle of the semigroup generated by a second order uniformly elliptic operator with bounded coefficients equals $\dfrac{\pi}{2}$ as a perturbation of the principal part which is similar to the Laplacian. Thus, our result allows to widen the applicability of the logarithmic convexity technique to the class of operators generating analytic semigroups of any angle $\psi \in \left(0,\dfrac{\pi}{2}\right]$, which include the case of unbounded drift terms of gradient type. The key idea relies on a general logarithmic convexity result due to Miller \cite{Mi'75}, combined with a developed sharp lower bound for the harmonic function appearing in the logarithmic convexity estimate.

In order to illustrate the applicability of our abstract results to various parabolic systems, we will consider an inverse initial data problem for the autonomous Ornstein–Uhlenbeck equation in $\mathbb{R}^N$, arising in stochastic perturbations of differential equations, quantum field theory and control theory. This equation exhibits most of the peculiarities that arise
with unbounded coefficients in terms of generation and observability questions. These equations have attracted considerable attention \cite{BP'18, CFMP'05, LRM'16, Lu'97, LMP'20, Me'01, MPP'02, MPRS'03}.

In \cite{Lu'97, MPRS'03}, Lunardi, Metafune et al. gave a precise description of the domain of Ornstein–Uhlenbeck operator respectively in $L^2$ and $L^p$ spaces with the invariant measure, while Chill et al. computed the sector of analyticity in \cite{CFMP'05}. Further generation and spectral results were established in \cite{Me'01, MPP'02, MPRS'03}. We refer to \cite{LMP'20} for a recent review on this topic. In contrast to parabolic equations in bounded domains, the observability question for such equations (on the whole space) is still at an earlier stage. Recently, Beauchard and Pravda-Starov \cite{BP'18, BP'17} have shown that the underlying equation is final state observable at any positive time from special observation regions, using an adaptation of the Lebeau-Robbiano strategy. As for inverse problems, to the best of our knowledge, the problem of stability in recovering initial data in parabolic equations with \textit{unbounded coefficients} has not been addressed in the literature.

The article is organized as follows. In Section \ref{sec2}, we start by recalling some preliminaries needed in the sequel. Then, in Section \ref{sec3}, we review some important results on the logarithmic convexity method which we use to prove our main result on the logarithmic stability in inverse initial data problems for final state observable systems. Finally, in Section \ref{sec4}, we apply our abstract result to the Ornstein–Uhlenbeck equation.

\section{Preliminaries \label{sec2}}
In this section, we recall some preliminaries that we will use in the sequel. We start with some facts from the theory of special functions, see \cite{OLBC'10} for more details.\\
\textbf{Incomplete Gamma function:}
the incomplete Gamma function is defined by
\begin{equation}
\Gamma(a,x):=\int_x^{\infty} t^{a-1} \mathrm{e}^{-t}\, \d t, \qquad a>0, \quad x \ge 0. \label{igm1}
\end{equation}
The (complete) Gamma function is given by
\begin{equation}
\Gamma(a)=\Gamma(a,0):=\int_0^{\infty} t^{a-1} \mathrm{e}^{-t}\, \d t, \qquad a>0. \label{gm1}
\end{equation}
In particular we have $\dfrac{\d}{\d x} \Gamma(a,x)=-x^{a-1} \mathrm{e}^{-x}$ for $a>0$ and $x\ge 0$.\\
\textbf{Incomplete Beta function:}
the incomplete Beta function is defined by
\begin{equation}
\mathrm{B}_x(a,b):=\int_0^x t^{a-1} (1-t)^{b-1}\, \d t, \qquad a>0, \quad b>0, \quad 0 \le x\le 1. \label{ibe1}
\end{equation}
The Beta function is simply
\begin{equation}
\mathrm{B}(a,b):=\mathrm{B}_1(a,b)=\int_0^1 t^{a-1} (1-t)^{b-1}\, \d t, \qquad a>0, \quad b>0. \label{be1}
\end{equation}
In particular, we have $\dfrac{\d}{\d x} \mathrm{B}_x(a,b)=x^{a-1} (1-x)^{b-1}$ for $0<x<1$. Moreover, the following identity holds
\begin{equation}
\mathrm{B}(a, 1-a)=\Gamma(a)\Gamma(1-a)=\dfrac{\pi}{\sin(\pi a)},  \qquad 0<a<1. \label{bepro}
\end{equation}
Also, for $0<a\le\frac{1}{2}$, we have
\begin{equation}
a \mathrm{B}_x(a, 1-a) \ge \left(\frac{1}{x}-1\right)^{\frac{1}{2}-a} \arcsin\sqrt{x}, \qquad 0<x \le 1. \label{inba}
\end{equation}

Now, we recall a few generalities from the theory of analytic semigroups. We denote by $\mathcal{L}(H)$ the Banach algebra of all linear bounded operators on $H$. For $\psi \in \left(0,\dfrac{\pi}{2}\right]$, let $\Sigma_\psi$ be the sector
$$\Sigma_\psi:= \{z\in \mathbb{C} \setminus \{0\}\colon |\arg z|< \psi \}.$$
\begin{definition}
Let $\psi \in \left(0,\dfrac{\pi}{2}\right]$. A family of operators $(T(z))_{z \in \Sigma_\psi \cup\{0\}} \subset \mathcal{L}(H)$ is called an analytic semigroup if
\begin{enumerate}[label=(\roman*),leftmargin=*]
\item $T(0)=I$ and $T(z_1+z_2)=T(z_1) T(z_2)$ for all $z_1, z_2 \in \Sigma_\psi$.
\item $\langle T(\cdot)x, y \rangle$ is analytic in $\Sigma_\psi$ for all $x,y \in H$.
\item $\lim\limits_{\Sigma_\varphi \ni z \rightarrow 0} T(z) x=x$ for all $x \in H$ and $\varphi \in (0,\psi)$.
\end{enumerate}
The angle of $(T(z))_{z \in \Sigma_\psi \cup\{0\}}$ is the supremum of possible values of $\psi$.
If, in addition,
\begin{enumerate}[label=(\roman*),leftmargin=*]
\item[(iv)] $\sup_{z\in \Sigma_\varphi} \|T(z)\| < \infty$ for all $\varphi \in (0,\psi),$
\end{enumerate}
we call $(T(z))_{z \in \Sigma_\psi \cup\{0\}}$ a bounded analytic semigroup.
\end{definition}

Let $A \colon D(A)\subset H \rightarrow H$ be the generator of an analytic semigroup $\left(\mathrm{e}^{tA}\right)_{t\ge0}$ of negative type, i.e., there exist $K, \kappa>0$ such that $\|\mathrm{e}^{tA}\| \le K \mathrm{e}^{-\kappa t}$ for all $t\ge 0$.
\begin{definition}
For $\alpha>0$, the bounded linear operator $(-A)^{-\alpha}$ is defined by
$$(-A)^{-\alpha}=\frac{1}{\Gamma(\alpha)} \int_0^\infty t^{\alpha-1} \mathrm{e}^{tA}\, \d t,$$
and by convention $(-A)^{0}=I$. For every $\alpha \ge 0$, the operator $(-A)^{-\alpha}$ is one-to-one. Then we define the closed linear operator $(-A)^{\alpha}=((-A)^{-\alpha})^{-1}$ with domain $D((-A)^{\alpha})=R((-A)^{-\alpha})$ endowed by the norm $\|x\|_{D((-A)^{\alpha})}=\|(-A)^{\alpha} x\|$ for every $x\in D((-A)^{\alpha})$.
\end{definition}
Next we collect some useful properties of $(-A)^\alpha$ that can be found in \cite{Paz'83}.
\begin{itemize}
\item For $\alpha \ge \beta >0$ we have $D((-A)^{\alpha}) \subset D((-A)^{\beta})$.
\item For $\alpha,\beta \in \mathbb{R}$ and $\gamma=\max(\alpha,\beta,\alpha +\beta)$,
\begin{equation}
(-A)^{\alpha +\beta}= (-A)^{\alpha} (-A)^{\beta} \qquad \text{ on } D((-A)^{\gamma}). \label{eqpw1}
\end{equation}
\item For every $t>0$ and $\alpha \ge 0$, $\mathrm{e}^{tA}\colon H \rightarrow D((-A)^{\alpha})$ and commutes with $(-A)^{\alpha}$ on $D((-A)^{\alpha})$. Moreover, the operator $(-A)^{\alpha} \mathrm{e}^{tA}$ is bounded and there exists a constant $M_\alpha>0$ such that
\begin{equation}
\|(-A)^{\alpha} \mathrm{e}^{tA}\| \le \dfrac{M_\alpha}{t^\alpha}. \label{eqpw2}
\end{equation}
\end{itemize}

\begin{remark}
The assumption that the type of $\left(\mathrm{e}^{tA}\right)_{t\ge0}$ be negative is not restrictive since by a scaling we can change $A$ to $A-\kappa$ so that its semigroup is of negative type.
\end{remark}

\section{Logarithmic convexity \label{sec3}}
The logarithmic convexity is one of the well-known methods that had been widely used to prove the conditional stability for improperly posed problems such as backward parabolic equations as well as inverse initial data problems. For the reader convenience, we refer to \cite{AN'63,ACM'21,GaT'11,Is'17,Mi'75,Pa'75,YZ'01}.

\subsection{Self-adjoint operator case}
In this special case, the logarithmic convexity result states that the norm of the solution to \eqref{E1} is logarithmically convex as a function of $t$.
\begin{lemma}
Consider a self-adjoint operator $A$ which is bounded above. Let $\theta >0$ and $M>0$ be fixed. For all $u_0\in H$ such that $\|u_0\|\leq M$, the solution of \eqref{E1} satisfies
\begin{equation}
\|u(t)\| \leq M^{1- \frac{t}{\theta}} \|u(\theta)\|^{\frac{t}{\theta}} \label{eqlc1}
\end{equation}
for all $0\leq t\leq \theta$.
\end{lemma}
A simple proof of this result relies on differentiating $\log \|u(t)\|$ twice with respect to $t$, with use of self-adjointness and Cauchy-Schwarz inequality, see, e.g., \cite{GaT'11}.

\subsection{General case: Analytic semigroup}
A general extension of logarithmic convexity inequality \eqref{eqlc1} for analytic semigroups (even in Banach spaces) was established in \cite{Mi'75}.

Assume that $A$ is the generator of a bounded analytic semigroup $\left(\mathrm{e}^{tA}\right)_{t\ge0}$ of angle $\psi \in \left(0,\dfrac{\pi}{2}\right]$ in $H$. In particular, if we set $\Sigma_\psi:= \{z\in \mathbb{C} \setminus \{0\}\colon |\arg z|< \psi \}$, then there exist $K\geq 1$ and $\kappa \ge 0$ such that
\begin{equation}
\|\mathrm{e}^{zA}\| \leq K \mathrm{e}^{\kappa \mathrm{Re}\,z} \quad \text{ on } \overline{\Sigma}_\psi. \label{A2}
\end{equation}

We shall use the following result from \cite{Mi'75}.
\begin{theorem}
Let $A$ be the generator of an analytic semigroup of angle $\psi$ with corresponding $K\geq 1$ and $\kappa \ge 0$ in \eqref{A2}. Let $\theta >0$ be fixed, $u_0\in H$ such that $\|u_0\|\leq M$, for some positive constant $M$, and $u$ be the solution of \eqref{E1}. Then
\begin{equation}
\|u(t)\| \leq K \mathrm{e}^{\kappa(t-\theta w(t))} M^{1- w(t)} \|u(\theta)\|^{w(t)}, \qquad 0\leq t\leq \theta, \label{eqlc2}
\end{equation}
where $w$ is the harmonic function on the strip
$$\mathcal{S}_\psi:= \{z\in \mathbb{C}\setminus \{0\} \colon |\arg z|< \psi \text{ and } |\arg (z-\theta)|> \psi\},$$
which is bounded and continuous on $\overline{\mathcal{S}}_\psi$, and takes the values $0$ and $1$ respectively on the left and right hand boundary half-lines of $\mathcal{S}_\psi$ (see Figure \ref{fig2}).
\end{theorem}
For the existence and uniqueness of such function $w$ we refer to \cite{Ke'67}. The proof of the above theorem follows from a basic inequality for analytic functions due to Carleman \cite{Ca'26}.

We start by the simple case when the strip $\mathcal{S}_\psi$ is rectangular.\\
\textbf{Special case:} $\psi=\dfrac{\pi}{2}$. In this case, $\mathcal{S}_{\frac{\pi}{2}}= \{z=t+\mathrm{i}s \colon 0<t< \theta \}$ and
\begin{equation}
w(z)=\dfrac{t}{\theta} \qquad \text{ for all } z=t+\mathrm{i}s\in \mathcal{S}_{\frac{\pi}{2}}. \label{eqw1}
\end{equation}

\begin{figure}[H]
\centering
\includegraphics{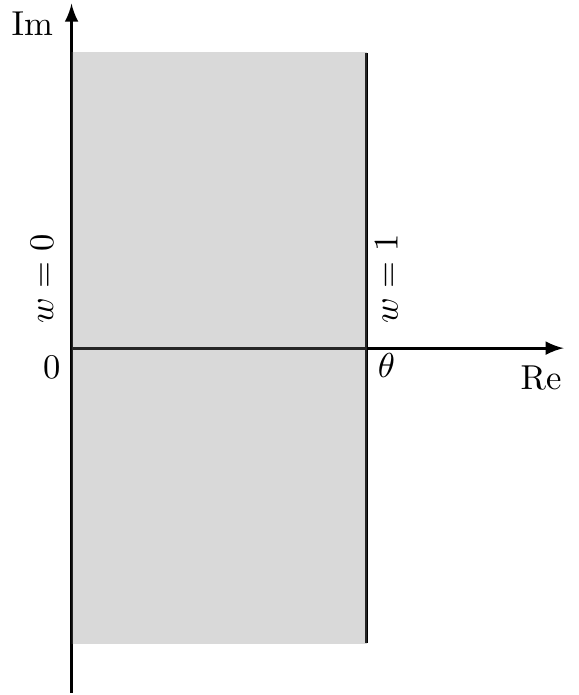}
\caption{The strip $\mathcal{S}_{\frac{\pi}{2}}$.}
\end{figure}

Consequently, we have the following result.
\begin{corollary}\label{corA4}
Let $\theta >0$ and $M>0$ be fixed and let $u_0\in H$ be such that $\|u_0\|\leq M$. If $A$ generates an analytic semigroup of angle $\dfrac{\pi}{2}$, then the solution of \eqref{E1} satisfies
\begin{equation}
\|u(t)\| \leq K M^{1- \frac{t}{\theta}} \|u(\theta)\|^{\frac{t}{\theta}}
\end{equation}
for all $0\leq t\leq \theta$.
\end{corollary}

\begin{remark}
This particular result allows us to weaken the regularity assumption on initial data. It was applied to a non-self adjoint case in \cite{ACM'21} where an inverse problem of potentials and initial data was considered. A related inverse source problem was treated in \cite{ACMO'20}.
\end{remark}

\noindent\textbf{General case:} $\psi \in \left(0,\dfrac{\pi}{2}\right]$. Denote by $\mathcal{S}_\psi^+$ the upper part of $\mathcal{S}_\psi$. That is, $$\mathcal{S}_\psi^+ :=\{z\in \mathcal{S}_\psi \colon \mathrm{Im}\, z >0\}.$$

\begin{figure}[H]
\centering
\includegraphics{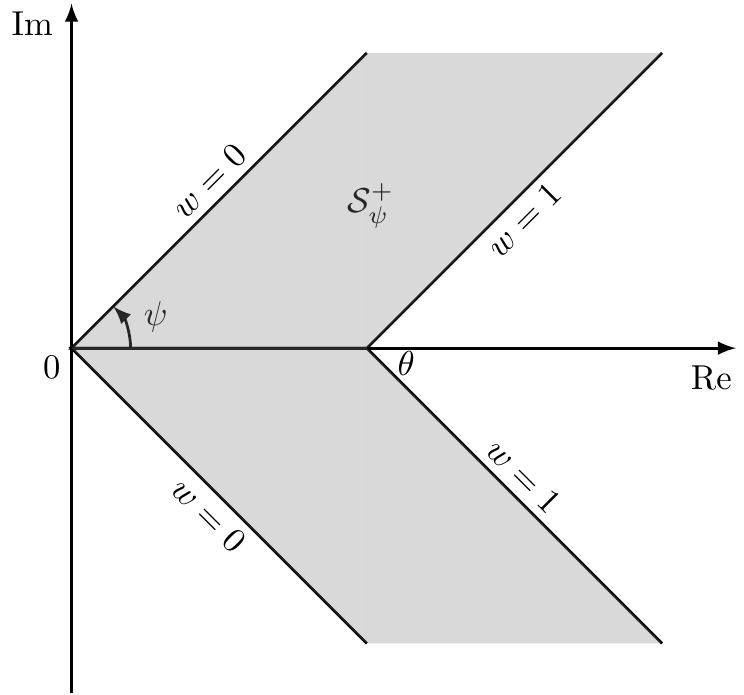}
\caption{The strip $\mathcal{S}_\psi$.} \label{fig2}
\end{figure}

\begin{lemma}
The harmonic function $w$ satisfies the following inequality
\begin{equation}
w(t)\ge \frac{2}{\pi} \left(\frac{\psi}{\sin\psi}\right)^{\frac{\pi}{2\psi}} \left(\frac{t}{\theta}\right)^{\frac{\pi}{2\psi}}, \qquad 0<t\le \theta.
\label{estw1}
\end{equation}
\end{lemma}

\begin{proof}
First, we explicit the harmonic function $w(t)$ for $0 \le t \le \theta$. For this aim, it suffices to obtain a conformal mapping from $\mathcal{S}_{\frac{\pi}{2}}^+$ onto $\mathcal{S}_\psi^+$ (by Schwarz reflection principle, see \cite[Chap. 2, \S 5.4]{SS'03}). The complex analytic function $g(z)=\theta \sin^2\left(\dfrac{\pi z}{2\theta}\right)$ is a conformal mapping from $\mathcal{S}_{\frac{\pi}{2}}^+$ onto the upper half-plane $\mathbb{C}^+ :=\{z\in \mathbb{C} \colon \mathrm{Im}\, z>0\}$ with $g(0)=0, g(\theta)=\theta$ and $g(\infty)=\infty$. A conformal mapping from $\mathbb{C}^+$ onto $\mathcal{S}_\psi^+$ is given by the Schwarz-Christoffel formula (see \cite[Chap. 8, \S 4]{SS'03}) with a change of variable as follows:
$$f(z)=C \theta\int_0^{\frac{z}{\theta}} \xi^{\frac{\psi}{\pi}-1} (1-\xi)^{-\frac{\psi}{\pi}} \, \d \xi,$$
where $C$ is a constant so that $f(\theta)=\theta$, that is
\begin{align*}
C &=\frac{1}{\mathrm{B}\left(\frac{\psi}{\pi}, 1-\frac{\psi}{\pi}\right)}=\frac{\sin\psi}{\pi},
\end{align*}
where we used the identity \eqref{bepro}. Hence, we obtain that
$$f(z)=\frac{\theta \sin\psi}{\pi} \int_0^{\frac{z}{\theta}} \xi^{\frac{\psi}{\pi}-1} (1-\xi)^{-\frac{\psi}{\pi}} \, \d \xi.$$
Hence, $h(z)=f\circ g(z)$ maps $\mathcal{S}_{\frac{\pi}{2}}^+$ conformally onto $\mathcal{S}_\psi^+$, and then the harmonic function on $\mathcal{S}_\psi^+$ is given by
\begin{equation}
w(z)=\dfrac{\mathrm{Re}(h^{-1}(z))}{\theta}. \label{eqw2}
\end{equation}
Set $\varphi=\dfrac{\psi}{\pi}$ and $h_0(x)=\mathrm{B}_{\sin^2(\frac{\pi x}{2\theta})}(\varphi, 1-\varphi)$, $0<x \le\theta$. The function $k(x)=\dfrac{h_0(x)}{x^{2\varphi}}$ is non-increasing. In fact, $k'(x)=\dfrac{x^{2 \varphi-1} (x h_0'(x)-2 \varphi h_0(x))}{x^{4\varphi}}$ for $0<x \le\theta$. Moreover, 
\begin{align*}
& x h_0'(x)-2 \varphi h_0(x)=\frac{\pi x}{\theta} \cot\left(\frac{\pi x}{2\theta}\right)^{1-2\varphi}-2\varphi \mathrm{B}_{\sin^2(\frac{\pi x}{2\theta})}(\varphi, 1-\varphi) \le 0, &&\quad 0<x\le\theta,\\
& \Longleftrightarrow \varphi \mathrm{B}_{\sin^2 t}(\varphi, 1-\varphi) - t\cot^{1-2\varphi}(t) \ge 0, &&\quad 0<t\le \frac{\pi}{2},\\
& \Longleftrightarrow \varphi \mathrm{B}_t(\varphi, 1-\varphi) - \left(\frac{1}{t}-1\right)^{\frac{1}{2}-\varphi} \arcsin\sqrt{t} \ge 0, &&\quad 0<t \le 1,
\end{align*}
which holds by the inequality \eqref{inba}. Using Hospital's rule we obtain
\begin{align*}
k(x) &\le \lim_{x \to 0^+} \frac{h_0(x)}{x^{2\varphi}}= \lim_{x \to 0^+} \frac{\mathrm{B}_{\sin^2(\frac{\pi x}{2\theta})}(\varphi, 1-\varphi)}{x^{2\varphi}}\\
&= \lim_{t \to 0^+} \frac{\pi^{2\varphi}}{4^\varphi} \theta^{-2\varphi}\frac{\mathrm{B}_{\sin^2 t}(\varphi, 1-\varphi)}{t^{2\varphi}}\\
&= \lim_{t \to 0^+} \frac{1}{\varphi} \frac{\pi^{2\varphi}}{4^\varphi} \theta^{-2\varphi} (t\cot t)^{1-2\varphi}\\
&= \frac{1}{\varphi}\frac{\pi^{2\varphi}}{4^\varphi} \theta^{-2\varphi}.
\end{align*}
Since $h(x)=\dfrac{\theta \sin \psi}{\pi} h_0(x) $, we infer that
\begin{equation}
h(x)\le \theta^{1-2\varphi} \dfrac{\sin\psi}{\psi} \dfrac{\pi^{2\varphi}}{4^\varphi} x^{2\varphi}, \qquad 0<x \le\theta. \label{esth1}
\end{equation}
The previous calculation implies that $h$ is increasing and then $h^{-1}((0,\theta])=(0,\theta]$. Consequently, we obtain that $h^{-1}(t) \ge \dfrac{2}{\pi} \left(\dfrac{\psi}{\sin\psi}\right)^{\frac{1}{2\varphi}} \theta^{1-\frac{1}{2\varphi}} t^{\frac{1}{2\varphi}}$ for all $t\in (0, \theta]$. Finally,
$$w(t)\ge \frac{2}{\pi} \left(\frac{\psi}{\sin\psi}\right)^{\frac{\pi}{2\psi}} \left(\frac{t}{\theta}\right)^{\frac{\pi}{2\psi}}$$
for all $t\in (0, \theta]$.
\end{proof}

\begin{remark}
When $\psi=\dfrac{\pi}{2}$ in the above proof, the function $g(z)=\theta \sin^2\left(\dfrac{\pi z}{2\theta}\right)$ is exactly the inverse function of $f(z)=\dfrac{\theta}{\pi} \displaystyle\int_0^{\frac{z}{\theta}} \xi^{-\frac{1}{2}} (1-\xi)^{-\frac{1}{2}} \, \d \xi=\dfrac{2 \theta}{\pi} \arcsin \sqrt{\dfrac{z}{\theta}}$ (with principal values). Then, we obtain that $h(z)=f\circ g(z)=z$ for all $z\in \mathcal{S}_{\frac{\pi}{2}}^+$. Substituting in \eqref{eqw2} we find again equality \eqref{eqw1}.
\end{remark}

Let us set the following two constants depending on the angle of analyticity of $\left(\mathrm{e}^{tA}\right)_{t\ge0}$,
\begin{equation}
\phi=\dfrac{\pi}{2\psi} \qquad  \text{ and } \qquad c_\psi=\dfrac{2}{\pi} \left(\dfrac{\psi}{\sin\psi}\right)^{\frac{\pi}{2\psi}}.
\end{equation}

\begin{remark}
The constant $c_\psi$ in inequality \eqref{estw1} is sharp, since for $\psi=\dfrac{\pi}{2}$, i.e., $\phi=c_\psi=1$, we obtain $w(t)=\dfrac{t}{\theta}$.
\end{remark}

Now, we are ready to state our main result on logarithmic stability estimate for the inverse initial data problem.
\begin{theorem}\label{logstab}
We assume that $u_{0} \in \mathcal{I}_{\epsilon, M}$ and the problem \eqref{E1} is final state observable in time $\theta$. Assume also that $p \in \left(1,\dfrac{1}{1-\epsilon}\right)$ and $s \in \left(0,1-\dfrac{1}{p}\right)$. Then, there exists a positive constant $K_1(M,\epsilon,K,\kappa,\kappa_{\mathrm{obs}},\theta,p,s)$ such that
\begin{equation}
\left\|u_0\right\|_H \le K_1 \left(\frac{\Gamma\left(\frac{1}{\phi}\right)-\Gamma\left(\frac{1}{\phi}, -c_\psi p\log \|\mathbf{C}u\|_{L^2(0,\theta;Y)}\right)}{\left(-c_\psi p\log \|\mathbf{C}u\|_{L^2(0,\theta;Y)}\right)^{\frac{1}{\phi}}\phi}\right)^{\frac{s}{p}}. \label{stab1}
\end{equation}
In particular,
\begin{equation}
\left\|u_0\right\|_H \le K_1 \left(\frac{\Gamma\left(\frac{1}{\phi}\right)}{\left(-c_\psi p\log \|\mathbf{C}u\|_{L^2(0,\theta;Y)}\right)^{\frac{1}{\phi}}\phi}\right)^{\frac{s}{p}}, \label{stab2}
\end{equation}
provided that $\|\mathbf{C}u\|_{L^2(0,\theta;Y)}$ is sufficiently small.
\end{theorem}

\begin{remark}
In the case when $\psi=\dfrac{\pi}{2}$, i.e., $\phi=c_{\psi}=1$, the stability estimate \eqref{stab1} becomes
$$\left\|u_0\right\|_H \le K_1 \left(\frac{\|\mathbf{C}u\|_{L^2(0,\theta;Y)}^p -1}{\log \|\mathbf{C}u\|_{L^2(0,\theta;Y)}}\right)^{\frac{s}{p}},$$
which is the same inequality obtained in \cite{GaT'11} for self-adjoint dissipative operators.
\end{remark}

\begin{proof}[Proof of Theorem \ref{logstab}]
We will adopt the ideas in \cite{GaT'11, YZ'08} to our case. Using inequality \eqref{eqlc2}, we have
$$\|u(t)\|_H \le K_0 \kappa_M \left(\frac{\|u(\theta)\|_H}{\kappa_M \mathrm{e}^{\kappa \theta}}\right)^{w(t)},$$
where $K_0=K \mathrm{e}^{\kappa \theta}$ and $\kappa_M>0$ denotes a constant which depends on $M$ (it might vary from line to line). Let $p>1$ and set $E=\left(\dfrac{\|u(\theta)\|_H}{\kappa_M \mathrm{e}^{\kappa \theta}}\right)^p$. Since $\|\mathbf{C}u\|_{L^2(0,\theta;Y)}$ is sufficiently small, we can assume that $0<E<1$. By inequality \eqref{estw1} and using a change of variable, we infer that
\begin{align*} \int_0^{\theta} \|u(t)\|_H^{p} \,\d t & \le (K_0 \kappa_M)^p \int_0^{\theta} E^{c_\psi (\frac{t}{\theta})^{\phi}}\, \d t\\
&= (K_0 \kappa_M)^p \theta \int_0^1 E^{c_\psi t^{\phi}}\, \d t\\
&= (K_0 \kappa_M)^p \theta \frac{\Gamma(\frac{1}{\phi})-\Gamma(\frac{1}{\phi}, -c_\psi\log E)}{(-c_\psi\log E)^{\frac{1}{\phi}}\phi}.
\end{align*}
Therefore,
\begin{align}
\|u\|_{L^p\left(0, \theta ; H\right)} \le K_0 \kappa_M\theta^{\frac{1}{p}}\left(\frac{\Gamma(\frac{1}{\phi})-\Gamma(\frac{1}{\phi}, -c_\psi\log E)}{(-c_\psi\log E)^{\frac{1}{\phi}}\phi}\right)^{\frac{1}{p}}. \label{e1}
\end{align}
By the property \eqref{eqpw1}, we have
$$
u'(t)=-(-A)^{1-\epsilon} \mathrm{e}^{t A}(-A)^{\epsilon} u_{0},
$$
and the inequality \eqref{eqpw2} yields
$$
\|u'(t)\|_H \le \frac{M_\epsilon}{t^{1-\epsilon}} \|u_0\|_{D\left((-A)^{\epsilon}\right)}.
$$
After integration we obtain
$$
\|u'\|_{L^p\left(0, \theta;H\right)} \le M \frac{M_\epsilon \theta^{\frac{1}{p}-(1-\epsilon)}}{(1-p(1-\epsilon))^{\frac{1}{p}}}.
$$
On the other hand, we have $\|u(t)\|_H \le K_0 \|u_0\|_H$ for all $t\in [0,\theta]$. Then,
$$\|u\|_{L^p\left(0, \theta ; H\right)} \le K_0 \kappa_M\theta^{\frac{1}{p}}.$$
Hence, for $p \in \left(1,\dfrac{1}{1-\epsilon}\right)$, we derive
\begin{align}
\|u\|_{W^{1, p}\left(0, \theta ; H\right)} \le K_0 \kappa_M \theta^{\frac{1}{p}}\left(1+\frac{M_\epsilon}{K_0\theta^{1-\epsilon}(1-p(1-\epsilon))^{\frac{1}{p}}}\right). \label{e2}
\end{align}
Let $0<s<1$. An interpolation inequality for \eqref{e1} and \eqref{e2} yields
\begin{align}
\|u\|_{W^{1-s, p}\left(0, \theta ; H\right)} \le C K_0 \kappa_M \theta^{\frac{1}{p}}\left(1+\frac{M_\epsilon}{K_0\theta^{1-\epsilon}(1-p(1-\epsilon))^{\frac{1}{p}}}\right)^{1-s} \nonumber\\
\times \left(\frac{\Gamma(\frac{1}{\phi})-\Gamma(\frac{1}{\phi}, -c_\psi\log E)}{(-c_\psi\log E)^{\frac{1}{\phi}}\phi}\right)^{\frac{s}{p}}. \label{intrp}
\end{align}
By virtue of the following embedding
$$
W^{1-s, p}\left(0, \theta ; H\right) \subset C\left(\left[0, \theta\right] ; H\right)
$$
for $s \in \left(0,1-\dfrac{1}{p}\right)$, we obtain
\begin{align*}
\|u_0\|_{H} \leqslant K_1\left(\frac{\Gamma(\frac{1}{\phi})-\Gamma(\frac{1}{\phi}, -c_\psi\log E)}{(-c_\psi\log E)^{\frac{1}{\phi}}\phi}\right)^{\frac{s}{p}}.
\end{align*}
By using the observability inequality in time $\theta$,
$$\|u(\theta)\|_H \le \kappa_{\mathrm{obs}} \|\mathbf{C}u\|_{L^2(0,\theta;Y)},$$
we deduce that
$$\left\|u_0\right\|_H \le K_1 \left(\frac{\Gamma\left(\frac{1}{\phi}\right)-\Gamma\left(\frac{1}{\phi}, -c_\psi p \log \left(\kappa_{\mathrm{obs}}\kappa_M^{-1} \mathrm{e}^{-\kappa \theta}\|\mathbf{C}u\|_{L^2(0,\theta;Y)}\right)\right)}{\left(-c_\psi p \log \left(\kappa_{\mathrm{obs}}\kappa_M^{-1} \mathrm{e}^{-\kappa \theta}\|\mathbf{C}u\|_{L^2(0,\theta;Y)}\right)\right)^{\frac{1}{\phi}}\phi}\right)^{\frac{s}{p}}.$$
It follows that,
$$\left\|u_0\right\|_H \le K_1 \left(\frac{\Gamma\left(\frac{1}{\phi}\right)-\Gamma\left(\frac{1}{\phi}, -c_\psi p \log \|\mathbf{C}u\|_{L^2(0,\theta;Y)}\right)}{\left(-c_\psi p \log \|\mathbf{C}u\|_{L^2(0,\theta;Y)}\right)^{\frac{1}{\phi}}\phi}\right)^{\frac{s}{p}}.$$
Indeed, for $c:=c_\psi p$ and $\sigma :=\kappa_{\mathrm{obs}}\kappa_M^{-1} \mathrm{e}^{-\kappa \theta}$, one can prove that the function
$$
r(x) :=\frac{\Gamma\left(\frac{1}{\phi}\right)-\Gamma\left(\frac{1}{\phi}, -c \log(\sigma x) \right)}{\left(-c \log(\sigma x)\right)^{\frac{1}{\phi}}\phi} \frac{\left(-c \log x\right)^{\frac{1}{\phi}}\phi}{\Gamma\left(\frac{1}{\phi}\right)-\Gamma\left(\frac{1}{\phi}, -c \log x \right)}
$$
is monotonic on $(0, \infty)$ (non-decreasing for $\sigma>1$ and non-increasing for $\sigma<1$).
\end{proof}

\begin{remark}
We emphasize that the result of Theorem \ref{logstab} yields also a logarithmic stability estimate for time-independent source terms $f \in H$ in the following system
\begin{empheq}[left = \empheqlbrace]{alignat=2}
\begin{aligned}
& u'(t)=A u(t) + f R(t), \quad t \in (0,\theta],\\
& u(0)=0,
\end{aligned}
\end{empheq}
from the observation $\mathbf{C}u'$, where $R\in H^1(0,\theta)$ is a known function which satisfies $R(0)\neq 0$. We refer to \cite{GaT'11} for more details.
\end{remark}

\section{Application to Ornstein–Uhlenbeck equation \label{sec4}}
In this section, we apply our abstract result to an inverse problem of determining a class of initial data in Ornstein–Uhlenbeck equation in $\mathbb{R}^N$ ($N\ge 1$ is an integer). Consider the following system
\begin{empheq}[left =\empheqlbrace]{alignat=2}
\begin{aligned}
&\partial_t y = \Delta y + B x\cdot \nabla y, \qquad 0<t<\theta , && x\in \mathbb{R}^N, \\
& y\rvert_{t=0}=y_0,   && x \in \mathbb{R}^N,\label{eq4.1}
\end{aligned}
\end{empheq}
where $B \neq 0$ is a real constant $N\times N$-matrix and $y_0$ is the initial condition.

The Ornstein–Uhlenbeck semigroup is defined by Kolmogorov's formula as follows
\begin{equation}
(T(t) f)(x)=\frac{1}{\sqrt{(4 \pi)^{N} \operatorname{det} Q_{t}}} \int_{\mathbb{R}^{N}} \mathrm{e}^{-\frac{1}{4}\left\langle Q_{t}^{-1} y, y\right\rangle} f\left(\mathrm{e}^{t B} x-y\right)\, \d y
\end{equation}
for every $t> 0$ and $f\in C_b(\mathbb{R}^{N})$, where $$Q_t=\int_0^t \mathrm{e}^{s B}\,\mathrm{e}^{s B^*}\, \d s,$$
with $B^*$ denotes the transpose matrix of $B$.

It is well known that the Ornstein–Uhlenbeck semigroup is strongly continuous in $L^2(\mathbb{R}^N)$ (with respect to the Lebesgue measure), but it is not analytic \cite{Me'01}. Therefore, a suitable space to study the Ornstein–Uhlenbeck semigroup is the weighted space $L^2_\mu :=L^2(\mathbb{R}^N,\d \mu)$, where $\mu$ will denote the unique invariant (probability) measure for $(T(t))_{t\ge 0}$, i.e.,
$$\int_{\mathbb{R}^{N}} T(t) f\, \d \mu=\int_{\mathbb{R}^{N}} f \,\d \mu$$
for every $t\ge 0$ and $f\in C_b(\mathbb{R}^{N})$. The existence of $\mu$ is equivalent to the property that the spectrum of the matrix $B$ lies in the open left half plane, that is,
\begin{equation}
\sigma(B) \subset \mathbb{C}_- :=\{z\in \mathbb{C} \colon \mathrm{Re}\,z <0\}. \label{spec}
\end{equation}
In this case, the invariant (Gaussian) measure $\mu$ is given by
$$\d \mu(x)=\frac{1}{\sqrt{(4 \pi)^{N} \operatorname{det} Q_{\infty}}} \mathrm{e}^{-\frac{1}{4}\left\langle Q_{\infty}^{-1} x, x\right\rangle}\, \d x =:\rho(x)\, \d x,$$
where $$Q_\infty:=\int_0^{+\infty} \mathrm{e}^{s B}\, \mathrm{e}^{s B^*}\, \d s.$$
We consider then the Ornstein–Uhlenbeck semigroup $(T(t))_{t\ge 0}$ in the space $L^2_\mu$. Its generator is the operator defined by
\begin{equation}
\begin{aligned}
D(A) & =H_{\mu}^{2}:=\left\{u \in L_{\mu}^{2}: D^{\alpha} u \in L_{\mu}^{2} \text { for }|\alpha| \le 2\right\},\\
A &:=\Delta + B x\cdot \nabla.
\end{aligned}
\end{equation}
See, for instance, \cite{Lu'97, MPRS'03}. We define $H_{\mu}^{s}$ for $s>0$, as follows
\begin{equation}
\begin{aligned}
& H_{\mu}^{s} :=\left\{f \in L_{\mu}^{2}: x \mapsto f(x) \mathrm{e}^{-\frac{1}{8}\left\langle Q_{\infty}^{-1} x, x\right\rangle} \in H^{s}\left(\mathbb{R}^{N}\right)\right\}, \\
&\|f\|_{H_{\mu}^{s}} :=\left\|f \mathrm{e}^{-\frac{1}{8}\left\langle Q_{\infty}^{2} x, x\right\rangle}\right\|_{H^{s}\left(\mathbb{R}^{N}\right)}. 
\end{aligned}\label{fdom}
\end{equation}

Let $I_N$ denote the identity matrix. The next result specifies the angle of analyticity.
\begin{theorem}[\cite{CFMP'05}]
The Ornstein–Uhlenbeck semigroup $(T(t))_{t\ge 0}$ in $L^2_\mu$ is analytic on the sector $\Sigma_{\psi}$, where $\psi \in \left(0,\dfrac{\pi}{2}\right]$ is defined by
$$\cot \psi=2\left\|\frac{1}{2} I_N + Q_{\infty} B^{*}\right\|.$$
Furthermore, $\psi$ is the optimal angle of analyticity.
\end{theorem}
Denoting $\gamma:=2\left\|\frac{1}{2} I_N + Q_{\infty} B^{*}\right\|$, we then have
\begin{equation}
\psi=\frac{\pi}{2}-\psi_{Q_{\infty} B^{*}},
\end{equation}
where $\psi_{Q_{\infty} B^{*}}:=\arctan \gamma$ is the spectral angle of $Q_{\infty} B^{*}$.

\begin{remark}
The operator $A$ is self-adjoint if and only if $Q_\infty B^*=BQ_\infty$ (see \cite{CG'02}). In this case, the identity $BQ_\infty + Q_\infty B^*=-I_N$ implies that $\dfrac{1}{2} I_N + Q_{\infty} B^{*}=0$, and then $\gamma=0$ and $\psi=\dfrac{\pi}{2}$. In general, the angle of analyticity of $(T(t))_{t\ge 0}$ in $L^2_\mu$ is strictly smaller than $\dfrac{\pi}{2}$. For example, if we consider the Ornstein–Uhlenbeck operator $A$ in $\mathbb{R}^2$ with $$B=\left(\begin{array}{cc}-1 & 2 \\ 0 & -1\end{array}\right),$$ then the spectral angle of $Q_{\infty} B^{*}$ is $\psi_{Q_{\infty} B^{*}}=\arctan 1=\dfrac{\pi}{4}$ (see \cite[Example 1]{CFMP'05}), and the angle of analyticity is $\psi=\dfrac{\pi}{2}-\psi_{Q_{\infty} B^{*}}=\dfrac{\pi}{4}$.
\end{remark}

Finally, let $\lambda>0$ be such that the semigroup of $A-\lambda$ is of negative type. Since $\lambda -A$ is a maximal accretive operator on $L^2_\mu$ (see, e.g., \cite{CFMP'05}), then by \cite[Theorem 2.30]{Ya'10} it has a bounded $\mathcal{H}^\infty $-calculus on $L^2_\mu$. \cite[Proposition 2.3]{Lu'97} allows us to characterize the domain of fractional powers $(\lambda-A)^\epsilon$ for $\epsilon \in (0,1)$, in terms of interpolation spaces. See \cite[Theorem 16.3]{Ya'10}.
\begin{lemma}
Let $[\cdot, \cdot]_\epsilon$ denote the complex interpolation functor. Then the following identity holds
\begin{align}
D((\lambda-A)^\epsilon)=[L^2_\mu,H^2_\mu]_\epsilon=H^{2\epsilon}_\mu, \qquad 0<\epsilon<1.
\end{align}
with equivalent norms.
\end{lemma}

Next, we recall the final state observability result for the system \eqref{eq4.1}. Let $\omega \subset \mathbb{R}^N$ be a nonempty open set. The following condition
\begin{equation}
\exists \delta, r>0, \forall y \in \mathbb{R}^N, \exists y' \in \omega, \quad B\left(y', r\right) \subset \omega \text { and }\left|y-y'\right|<\delta, \label{obsreg}
\end{equation}
is sufficient for final state observability of the system \eqref{eq4.1} to hold from the observation region $\omega$ at any positive time $\theta>0$. The property \eqref{obsreg} means that the subdomain $\omega$ is sufficiently stretched out in $\mathbb{R}^N$ \cite{LRM'16}.

We consider the observation distributed on the region $\omega$, which is given by the following observation operator $$\mathbf{C}u(t,x):=\mathds{1}_\omega(x)u(t,x),$$
which is admissible. We shall use the following final state observability result for system \eqref{eq4.1}, recently proved in \cite[Corollary 1.7]{BP'18}.
\begin{proposition}
Let $\theta>0$ be fixed and $\omega \subset \mathbb{R}^N$ be an open set satisfying \eqref{obsreg}. Consider $y$ the mild solution of \eqref{eq4.1}. If the spectral condition \eqref{spec} holds, then there exists a positive constant $\kappa_\theta=\kappa_\theta(\omega,\theta)$ such that for all $y_0 \in L^2_\mu$, we have
\begin{equation}
\|y(\theta,\cdot)\|_{L^2_\mu}^2 \leq \kappa_\theta \int_0^\theta \|y(t,\cdot)\|_{L^2_\mu(\omega)}^2\,\d t, \label{obsineq}
\end{equation}
where $L^2_\mu(\omega):=L^2(\omega,\d \mu)$. 
\end{proposition}

The following result of logarithmic stability is an application of Theorem \ref{logstab}.
\begin{proposition}\label{prop4.4}
Let $p \in \left(1,\dfrac{1}{1-\epsilon}\right)$ and $s \in \left(0,1-\dfrac{1}{p}\right)$. If the spectral condition \eqref{spec} holds, then there exists a positive constant $K_1(M,\epsilon,K,\kappa,\kappa_\theta,\theta,p,s)$ such that, for all $y_0 \in H_\mu^{2\epsilon}$ with $\left\|y_0\right\|_{H_\mu^{2\epsilon}} \le M$, we have
\begin{equation}
\|y_0\|_{L^2_\mu} \le K_1 \left(\frac{\Gamma\left(\frac{1}{\phi}\right)-\Gamma\left(\frac{1}{\phi}, -c_\psi p\log \|y\|_{L^2\left(0,\theta ; L^2_\mu(\omega)\right)}\right)}{\left(-c_\psi p\log \|y\|_{L^2\left(0,\theta ; L^2_\mu(\omega)\right)}\right)^{\frac{1}{\phi}}\phi}\right)^{\frac{s}{p}},
\end{equation}
where $y$ is the solution of system \eqref{eq4.1}. Moreover, we have
\begin{equation}
\|y_0\|_{L^2_\mu} \le K_1 \left(\frac{\Gamma\left(\frac{1}{\phi}\right)}{\left(-c_\psi p\log \|y\|_{L^2\left(0,\theta ; L^2_\mu(\omega)\right)}\right)^{\frac{1}{\phi}}\phi}\right)^{\frac{s}{p}},
\end{equation}
provided that $\|y\|_{L^2\left(0,\theta ; L^2_\mu(\omega)\right)}$ is sufficiently small.
\end{proposition}

\begin{remark}
We close the paper with the following remarks:

a) The result of Proposition \ref{prop4.4} holds true for the general Ornstein–Uhlenbeck operator given by
$$A=\dv(Q \nabla) + B x \cdot \nabla,$$
where $Q=Q^*$ is a real constant $N\times N$-matrix which is positive definite, provided the following changes are made:
$$Q_{\infty}=\int_{0}^{+\infty} \mathrm{e}^{s B}\, Q\, \mathrm{e}^{s B^{*}}\, \d s,$$
$$\gamma:=\cot \psi=2\left\|\frac{1}{2} I_N +Q^{-1 / 2} Q_{\infty} B^{*} Q^{-1 / 2}\right\|.$$

b) In Proposition \ref{prop4.4}, we considered initial data in the weighted space $L^2_\mu$ with the invariant measure $\mu$, since we used $L^2_\mu$ final state observability \eqref{obsineq}. We notice that the Ornstein–Uhlenbeck equation is also finale state observable in $L^2(\mathbb{R}^N)$ (see \cite[Theorem 1.3]{BP'18}). However, our approach does not allow one to conclude a stability result for a class of initial data in $L^2(\mathbb{R}^N)$ due to the lack of analyticity of the semigroup. It would be of much interest to prove a similar result with an observation in $L^2(\mathbb{R}^N)$ norm for initial data in the admissible set
$$\mathcal{I}_{\epsilon,M}:=\{u_0\in H^{2\epsilon}(\mathbb{R}^N) \colon \|u_0\|_{H^{2\epsilon}(\mathbb{R}^N)} \le M\}.$$

c) We stress that our result in Theorem \ref{logstab} is still valid even in Banach spaces, with final state observability in $L^r(0,\theta; Y)$ ($r\in [1, \infty]$). For instance, even for a self-adjoint operator associated with a sub-Markovian semigroup, its realization on $L^p$ generates an analytic semigroup on the sector of angle $\psi_p:=\frac{\pi}{2}-\arctan \frac{|p-2|}{2 \sqrt{p-1}}$ for $p \in(1, \infty)$, see \cite[Theorem 3.13]{Ou'04}. Nevertheless, we restricted ourselves to Hilbert spaces for the sake of simplicity and compatibility with the application we consider.
\end{remark}

\section{Conclusion}
In this paper, we have considered an inverse problem for initial datum in a general class of parabolic systems. A general logarithmic convexity result is investigated. In particular, an explicit form of the harmonic function appearing in the logarithmic convexity inequality is obtained depending on the analyticity angle of the semigroup associated to the system. Then, a sharp lower bound for the harmonic function is derived. Based on the final state observability of the system, a conditional logarithmic stability for the initial data is established. Finally, the abstract result has been illustrated by considering the Ornstein–Uhlenbeck equation in $\mathbb{R}^N$, which is the prototype of equations with unbounded gradient coefficients.

\end{document}